\documentclass[11pt]{article}

\usepackage{amsmath, amsthm, amsopn, amssymb, enumerate}

\newcounter{defc}

\newtheorem{thm}{Theorem}
\newtheorem{lemma}[thm]{Lemma}
\newtheorem{prop}[thm]{Proposition}
\newtheorem{cor}[thm]{Corollary}
\newtheorem{claim}[thm]{Claim}

\numberwithin{thm}{section}

\theoremstyle{definition}
\newtheorem{defn}[defc]{Definition}

\newcommand{\dist}{\mathrm{dist}}
\newcommand{\diam}{\mathrm{diam}}

\newcommand{\Hom}{\mathsf{Hom}}
\newcommand{\Lip}{\mathrm{Lip}}

\newcommand{\R}{\mathbb{R}}
\newcommand{\Z}{\mathbb{Z}}

\newcommand{\A}{\mathcal{A}}
\newcommand{\C}{\mathcal{C}}
\newcommand{\cP}{\mathcal{P}}
\renewcommand{\P}{\mathbb{P}}
\renewcommand{\Pr}{\mathbb{P}}
\renewcommand{\S}{\mathcal{S}}
\newcommand{\E}{\mathbb{E}}
\newcommand{\n}{N}
\renewcommand{\b}{\partial}

\newcommand{\T}{\mathbb{T}}

\newcommand{\Var}{\mathop \mathrm{Var}}

\DeclareMathOperator{\phase}{phase}

\title{Lipschitz Functions on Expanders are Typically Flat}

\author{Ron Peled\thanks{School of Mathematical Sciences, Tel Aviv University, Tel Aviv, Israel. E-mail address: {\tt peledron@post.tau.ac.il}. Supported by an ISF grant and an IRG grant.}
\and Wojciech Samotij\thanks{School of Mathematical Sciences, Tel
Aviv University, Tel Aviv, Israel; and Trinity College,
Cambridge~CB2~1TQ, UK. E-mail address: {\tt ws299@cam.ac.uk}.
Research supported in part by ERC Advanced Grant DMMCA and a~Trinity
College JRF.} \and Amir Yehudayoff\thanks{Department of Mathematics,
Technion-IIT, Haifa, Israel. E-mail address: {\tt
amir.yehudayoff@gmail.com}. Horev fellow -- supported by the Taub
Foundation. Research supported by grants from ISF and BSF.} }

\begin{document}

\date{}

\maketitle

\begin{abstract}
This work studies the typical behavior of random integer-valued
Lipschitz functions on expander graphs with sufficiently good
expansion. We consider two families of functions: $M$-Lip\-schitz
functions (functions that change by at most $M$ along edges) and
integer-ho\-mo\-mor\-phisms (functions that change by exactly $1$ along
edges). We prove that such functions typically exhibit very small
fluctuations. For instance, we show that a uniformly chosen
$M$-Lipschitz function takes only $M+1$ values on most of the graph,
with a double exponential decay for the probability to take other
values.
\end{abstract}

\section{Introduction}

In this work we investigate the typical behavior of random Lipschitz
functions on expander graphs. We focus on the following two models:
An \emph{$M$-Lipschitz function} on a graph $G$ is an integer-valued
function on the vertices of $G$ which changes by at most $M$ between
adjacent vertices. Similarly, a \emph{$\Z$-homomorphism} (or simply a
\emph{homomorphism}) on $G$ is an integer-valued function on the vertices
of $G$ which changes by exactly one between adjacent vertices.
Taking the graph $G$ to be a finite expander graph, we consider the
typical properties of functions chosen uniformly at random from one
of these families (fixing the function value to be zero at some
fixed vertex).

One motivation for this work comes from previous investigations of
$\Z$-homomorphisms on several tree-like graphs \cite{BenHagMos}, on
the hypercube~\cite{Gal, Kah2} and on (finite boxes of) the
lattice $\Z^d$ for large $d$ \cite{Pel}. These suggest that typical
Lipschitz functions on highly connected graphs tend to exhibit very
small fluctuations. Indeed, such behavior can also be expected from
a comparison with random surface models in statistical mechanics,
such as the Gaussian free field (see also \cite{BenSch}). Expander
graphs are natural candidates to test this paradigm, as they are
highly connected graphs which are important for many applications,
see the survey~\cite{HoLiWi}. It is also well-known that most graphs
are expanders.


Additional motivation is to try and understand to what extent the local
behavior of typical Lipschitz functions on a graph is affected by
the global features of the graph. Many expander graphs locally have
the structure of a tree, though globally they are very far from
being a tree. Will the typical Lipschitz function on an expander
graph exhibit locally large fluctuations, as it does on a tree, or
will it exhibit locally small fluctuations, as suggested by the
global structure?

Our results apply to expander graphs with sufficiently good
expansion, in a certain quantitative sense (see
Definitions~\ref{def: good} and \ref{def: good bi}). We show that on
such graphs, the typical Lipschitz function exhibits very small
fluctuations. Thus, although the graph may look like a tree locally,
it is the global structure of it which determines the local
fluctuations. More specifically, we find that a random $M$-Lipschitz function will
take only $M+1$ different values on most of the graph. Moreover, the
probability that at any fixed vertex the function takes a value
which is further than $tM$ from this set of $M+1$ values decays like a
double exponential in $t$. As a result, the maximum value of the
function is, with high probability, of order $\log(\log n)$, where
$n$ is the number of vertices of the graph\footnote{Logarithms
in this text are of base $e$.}.

Similar results are obtained for $\Z$-homomorphisms, where it is
required that the underlying graph is bipartite. There, we find that
the typical function takes predominantly one value on one of the
color classes and two values on the other color class, again with a
double exponential decay of the probability to take other values. In
a somewhat different setting, our methods yield also that
\emph{grounded Lipschitz functions} on $d$-ary trees, i.e.,
functions constrained to take the value zero on all leaves of the
tree, exhibit similarly small fluctuations.

While our methods require the graph to have sufficiently good
expansion, it is not clear to what extent is this requirement
necessary. A discussion with several open questions is presented in
Section~\ref{sec: sum}.

\subsection{$M$-Lipschitz functions}

\label{sec:Lipschitz-functions}

$M$-Lipschitz functions on graphs are defined as follows.

\begin{defn}[$M$-Lipschitz functions]
Let $v_0$ be a fixed vertex in a graph $G$ and let $M$ be a
positive integer. Denote by $\Lip_{v_0}(G;M)$ the family of
$M$-Lipschitz functions from the vertex set of $G$ to $\Z$
that send $v_0$ to the origin, that is, the family of maps $f \colon V(G)
\to \Z$ such that $f(v_0) = 0$ and $|f(u)-f(v)| \leq M$ for every
$\{u,v\} \in E(G)$.
\end{defn}
We note for later use that if $G$ is a finite connected graph, then
$\Lip_{v_0}(G;M)$ is a finite set.

There are several equivalent definitions of expander graphs. Our definition is inspired
by the so-called expander mixing lemma (see~\cite{HoLiWi}), which relates
the edge distribution of a regular graph to the spectral properties of its adjacency
matrix. For two subsets $A, B$ of the vertices of a graph $G$, denote by $E(A,B)$ the set
of pairs $(a, b) \in  A \times B$ such that $\{a,b\}$ is an edge of $G$
and denote $e(A,B) = |E(A,B)|$.

\begin{defn}[Expander]
  \label{defn:expander}
  A $d$-regular $n$-vertex graph $G$ is called a {\em $\lambda$-expander} if for all $S,T \subseteq V(G)$ we have
  \[
  \left| e(S,T) - \frac{d}{n}|S||T| \right| \leq  \lambda \sqrt{|S||T|}.
  \]
\end{defn}

Every $d$-regular graph is an expander with $\lambda=d$ and hence
the definition becomes meaningful only for $\lambda<d$. Simple
examples show that expanders cannot be too good: for $S = \{v\}$ and
$T$ the $d$ neighbors of $v$ in $G$ we obtain $\lambda \geq \sqrt{d}
(1- d/n)$, and for $S = T = \{v\}$ we have $\lambda \geq d/n$ so
that in particular, $\lambda \geq 1/2$.

In our theorems, we consider sufficiently good expanders in the
sense that $\lambda$ is required to be smaller than some specific
function of $d$ and $M$. It can be shown that for $d>2$, most
$d$-regular graphs are good expanders. For example, Friedman showed
that for every $d > 2$ and every $\varepsilon > 0$, with probability
tending to $1$ as $n$ tends to infinity, a uniformly chosen random
$n$-vertex $d$-regular graph is a $(2\sqrt{d-1} +
\varepsilon)$-expander \cite{Fr}.

\begin{defn}[$M$-good expander]
\label{def: good} A graph $G$ is an {\em $M$-good expander} if $G$
is a $d$-regular $\lambda$-ex\-pan\-der with
$\lambda \leq \frac{d}{32(M+1)\log (9Md^2)}$.
\end{defn}

Our main result shows that a \emph{typical} $M$-Lipschitz function
on an $M$-good expander is locally very flat. In particular, such a function
takes values in a set of $M+1$ consecutive integers at all but an
\emph{exponentially} small fraction of the vertices, where
exponentially small is with respect to the parameters of the
expander. The first step towards establishing this property is to
note that \emph{every} $M$-Lipschitz function takes values in a set
of $M+1$ consecutive integers at all but a \emph{polynomially} small
fraction of the vertices.

\begin{lemma}
  \label{lemma:phase}
  Let $G$ be an $n$-vertex $d$-regular $\lambda$-expander and $v_0\in V(G)$. For every $f\in \Lip_{v_0}(G;M)$,
  there exists $k\in\Z$ such that
  \begin{equation}\label{eq:phase_def}
    \left| \big\{v \colon  f(v)\notin\{k,k+1,\ldots, k+M\} \big\} \right| \le \frac{2 \lambda n}{d} .
  \end{equation}
  Moreover, there is a way to associate to each $f\in \Lip_{v_0}(G;M)$ which is
  not identically zero, an interval of the form $\{k,k+1,\ldots,
k+M\}$ satisfying \eqref{eq:phase_def}, which we denote by
$\phase(f)$, in such a manner that $\phase(-f) = - \phase(f)$.
\end{lemma}
The lemma allows us to define $\phase(f)$ for $M$-Lipschitz functions
which are not identically zero. For completeness, we define the
phase of the zero function as $\{0\}$. The lemma does not provide
information about the range of values which the function can take.
Indeed, the diameter of an $n$-vertex expander graph $G$ is of the
order of $\log n$ (see also Corollary~\ref{cor:exp-diam} below) and
hence there exist Lipschitz functions on $G$ taking order $\log n$
distinct values (though they still must satisfy
Lemma~\ref{lemma:phase}). The next theorem and corollary, which are
our main results, show that such large fluctuations are highly
atypical.
%

To discuss the typical properties of an $M$-Lipschitz function, we
introduce the uniform probability measure on $\Lip_{v_0}(G;M)$.
Restricting to finite connected graphs $G$, we denote by $f \in_R
\Lip_{v_0}(G;M)$ a uniform random element in $\Lip_{v_0}(G;M)$. To
describe our results, we use the metric of $G$ (the distance between
a pair $(u, v)$ of vertices is the length of a shortest path connecting
$u$ and $v$ in $G$) and we denote by
$B(v,t)$ the ball of radius $t$ around a vertex $v$. Finally, for an
integer $b$ and a non-empty set $A \subseteq \Z$, define $\dist(b,A)
= \min_{a \in A} |b - a|$.

\begin{thm}\label{thm:phase_dev}
Let $G$ be a connected $M$-good expander and $v_0\in V(G)$. Let
$f\in_R \Lip_{v_0}(G;M)$. For every positive integer $t$ and every
$v\in V(G)$,
\begin{equation}\label{eq:phase_dev_estimate}
  \P \big( \dist(f(v),\phase(f))> (t-1)M \big)\, \le\,
  \exp\left(-\frac{|B(v,t)|}{5(M+1)}\right).
\end{equation}
\end{thm}

Let us make a few comments:
\begin{enumerate}[(1)]
\item Proposition~\ref{prop:exp-balls-growth} below shows that on
any expander, $|B(v,t)|$ grows at least exponentially in $t$. Thus,
under the conditions of Theorem~\ref{thm:phase_dev}, we have
\begin{equation*}
  \P \big( \dist(f(v),\phase(f))> (t-1)M \big)\, \le\, \exp\left(-C^t\right)
\end{equation*}
for some (explicit) constant $C>1$ depending only on $d$ and
$\lambda$. If the girth of $G$ is at least $2t$ we have $|B(v,t)|\ge
d(d-1)^{t-1}$ which can improve the estimate further.
\item
  If we apply Theorem~\ref{thm:phase_dev} with $v=v_0$, we obtain bounds on the probability distribution of $\phase(f)$.
\item
  As a simple consequence of Theorem~\ref{thm:phase_dev}, we see that the variance of $f(v)$ is uniformly bounded for all $v\in V(G)$ by $C_1 M^2$, where $C_1$ is a function of only $d$ and $\lambda$.
\end{enumerate}

We can use Theorem~\ref{thm:phase_dev} to obtain a bound on the
maximum of $f$ (or alternatively, on the range of values $f$ takes).
By applying a union bound over all vertices of $G$ we obtain the
following estimate.
\begin{cor}
Let $G$ be a connected $M$-good expander and $v_0\in V(G)$. Let
$f\in_R \Lip_{v_0}(G;M)$. Then there exists a constant $C$ depending
only on $d$ and $\lambda$ such that
\begin{equation*}
  \P \big( \max(f) > CM\log(\log n) \big)\le \frac{1}{n^2}.
\end{equation*}
In particular, $\E[\max(f)] \le 2CM\log(\log n)$.
\end{cor}

It appears that by applying the techniques of Benjamini et
al.~\cite{BenYadYeh} one may obtain a converse to the above
corollary, showing that $\E[\max(f)] \ge cM\log\log n$ for some $c$
depending only on $d$ and $\lambda$. We do not explore this
direction here, but mention that the general approach is to consider
the possibility that $f$ grows quickly in a ball of small radius
and then use the fact that there are many such disjoint
balls.


\subsection{Homomorphisms}
We also study the typical behavior of graph homomorphisms from $G$
to $\Z$ or, as we will refer to them in the sequel,
\emph{homomorphism height functions} or simply \emph{homomorphisms}.

\begin{defn}[Homomorphisms]
Let $v_0$ be a fixed vertex in the graph $G$. Denote by
$\Hom_{v_0}(G)$ the family of graph homomorphism from $G$ to $\Z$
that send $v_0$ to the origin, that is, the family of maps $f : V(G)
\to \Z$ so that $f(v_0) = 0$ and $|f(u)-f(v)| = 1$ for every
$\{u,v\} \in E(G)$.
\end{defn}

The family $\Hom_{v_0}(G)$ is non-empty if and only if $G$ is
bipartite. It is finite if and only if $G$ is finite, connected and
bipartite. Let $G$ be bipartite with color classes $V_0$ and $V_1$
and assume that $v_0$ is in $V_0$. Then, every $f \in \Hom_{v_0}(G)$
takes even values on vertices of $V_0$ and odd values on $V_1$.

Since bipartite graphs contain very large independent sets, they clearly are not expanders in the sense of Definition~\ref{defn:expander}. Therefore, we shall consider the following
(standard) bipartite analogue of expander graphs.

\begin{defn}[Bi-expander]
  \label{defn:bi-expander}
  A $d$-regular bipartite graph $G$ with $2n$ vertices\footnote{Since every regular bipartite graph has an even number of vertices, we choose to denote this number by $2n$. This is somewhat inconsistent with our notation from Section~\ref{sec:Lipschitz-functions}, but will make later formulas look considerably cleaner.} is called a {\em $\lambda$-bi-expander} if the following holds. Let $V_0$ and $V_1$
be the two color classes of $G$. For all $S \subseteq V_0$ and $T
\subseteq V_1$,
$$  \left| e(S,T) - \frac{d}{n}|S||T| \right| \leq  \lambda \sqrt{|S||T|}.$$
\end{defn}

Similarly as in the case of (non-bipartite) $\lambda$-expanders, every $d$-regular bipartite graph is a $\lambda$-bi-expander with $\lambda = d$ and hence Definition~\ref{defn:bi-expander} becomes meaningful only for $\lambda < d$. In our theorems, we consider sufficiently good expanders, in the sense that $\lambda$ is required to be smaller than some function of $d$. Similarly as in the non-bipartite case, it can be shown that most $d$-regular bipartite graphs are good bi-expanders.

\begin{defn}[Good bi-expander]
  \label{def: good bi}
  A graph $G$ is a {\em good} bi-expander if $G$ is a $d$-regular
  $\lambda$-bi-expander with $\lambda  \leq \frac{d}{300 \log  d}$.
\end{defn}

We will show that homomorphism height functions on good bi-expanders
are very flat in the sense that a typical function is nearly
constant on one of the two color classes. Here, nearly constant
will mean constant on all but an \emph{exponentially} (in the
parameters of the expander) small fraction of the vertices. As in the
$M$-Lipschitz case, we first note that \emph{every} homomorphism
height function is constant on one of the two color classes, apart from
a \emph{polynomially} small fraction of the vertices.

\begin{lemma}
  \label{lemma:phase_hom}
  Let $G$ be a $2n$-vertex $d$-regular
  $\lambda$-bi-expander and $v_0\in V(G)$. For every $f\in \Hom_{v_0}(G)$,
  there exist $i\in\{0,1\}$ and $k\in\Z$ such that
  \begin{equation}\label{eq:phase_def_hom}
    \big| \{v\in V_i \colon  f(v)\neq k\} \big| \le \frac{2\lambda n}{d} .
  \end{equation}
  Let $i^* \in \{0,1\}$ be the smallest $i$ so that there exists a $k$ satisfying \eqref{eq:phase_def_hom} with $V_i$.
  Denote by
  $\phase(f)$ the
  smallest $k$ satisfying \eqref{eq:phase_def_hom} for $i^*$.
  The parity of $\phase(f)$ is $i^*$.
  Moreover, if $\lambda < d/3$, then
  \begin{equation}
    \label{eq:phase_hom-cor}
    \big| \{v \in V(G) \colon |f(v) - \phase(f)| \ge 2\} \big| \le \frac{3\lambda n}{d}.
  \end{equation}
\end{lemma}

As before, we restrict attention to finite, connected and bipartite
$G$ and let $f \in_R \Hom_{v_0}(G)$ denote a uniform random element
in $\Hom_{v_0}(G)$. The following theorem is our main result for
homomorphism height functions.

\begin{thm}\label{thm:phase_dev_hom}
  Let $G$ be a connected good bi-expander and $v_0\in V(G)$. Let $f\in_R
  \Hom_{v_0}(G)$. For every integer $t \ge 1$ and every $v\in V(G)$,
  \[
  \P(|f(v)-\phase(f)|> t)\le \exp\left(-\frac{|B(v,t)|}{3}\right).
  \]
\end{thm}
Similar comments as after Theorem~\ref{thm:phase_dev} apply here as
well: Proposition~\ref{prop:bi-exp-balls-growth} shows that
$|B(v,t)|$ grows at least exponentially in $t$ and hence, under the
conditions of the theorem,
\begin{equation*}
  \P(|f(v)-\phase(f)|> t)\le \exp\left(-C^{t}\right),
\end{equation*}
where $C>1$ is a constant depending only on $d$ and $\lambda$.
Additionally, by taking $v=v_0$ we may obtain bounds on the
distribution of $\phase(f)$ and so the theorem implies that
$\max_{v\in V(G)} \Var f(v)$ is bounded above by a constant
depending only on $\lambda$ and $d$.

Observing that the diameter of $2n$-vertex bi-expanders is
logarithmic in $n$, see Corollary~\ref{cor:bi-exp-diam}, we see that
every homomorphism height function can take at most order $\log n$
values. Our next statement sharpens this bound considerably by
showing that the typical order of magnitude of the range of the
random homomorphism height function is $\log(\log n)$. The lower
bound in this theorem follows from the results of \cite{BenYadYeh}.
\begin{thm} \label{thm: loglog range}
Let $G$ be a good $2n$-vertex bi-expander and let $v_0\in V(G)$. There exist
constants $c,C > 0$, which depend only on $d$ and $\lambda$, such that
if $f\in_R \Hom_{v_0}(G)$, then
$$\Pr [ c \log  \log  n \leq \max(|f|) \leq C \log  \log  n ] \ge 1 - \frac{1}{n^2}.$$
\end{thm}

\subsection{Grounded Lipschitz functions on
trees}\label{sec:grounded}

As a final example of the applicability of our methods, we study
Lipschitz functions on trees. We denote by $\T_h^d$ the complete
$(d-1)$-ary tree of height $h\ge 1$ (in which all internal vertices, including
the root, have degree $d$), denote its root vertex by $v_r$ and the set of its
leaves by $V_L$. By definition, $\T_h^d$ is a tree with
$d(d-1)^{h-1}$ leaves, all at distance exactly $h$ from $v_r$. It is
a simple observation that a uniformly chosen function from
$\Hom_{v_r}(\T_h^d)$ will behave like a random walk along branches of the
tree and thus typically take values of order $\sqrt{h}$ at the leaves of the
tree and reach a value of order $h$ at its maximum; similarly, a uniformly chosen
function from $\Lip_{v_r}(\T_h^d;M)$ typically takes values of order $M\sqrt{h}$
at the leaves of the tree and reaches a value of order $Mh$ at its maximum.
We consider instead a different
probabilistic model. We let $\Lip_{V_L}(\T_h^d;M)$ be the set of
$M$-Lipschitz functions on $\T_h^d$ which equal zero on $V_L$, and
similarly $\Hom_{V_L}(\T_h^d)$ be the set of homomorphism height
functions which equal zero on $V_L$. We call such functions
\emph{grounded}. Benjamini et al. \cite{BenHagMos} investigated
typical grounded homomorphism functions on $\T_h^d$ and showed that
they exhibit very different behavior from homomorphism functions
normalized at the root of the tree. For example, they showed that
for any $d>2$, if $f\in_R \Hom_{V_L}(\T_h^d)$ then
\begin{equation*}
  \P(|f(v_r)| > t) \le 2c^{d^{t}}
\end{equation*}
where $c=c(d)\le 2^{-d+1}$, so that, in particular, the distribution
of $f(v_r)$ is tight as $h\to\infty$. Using the same techniques as
for expander graphs, we establish similar results for
grounded $M$-Lipschitz functions on trees. Also, in the next section,
we explain our argument briefly in the context of
grounded homomorphism functions. Our method seems simpler than that
of \cite{BenHagMos} but has the disadvantage that it does not apply
for small $d$.
\begin{thm}\label{thm:grounded}
  For any $M\ge 1$, $d>40(M+1)\log(M+1)$ and $h\ge 1$, if $f\in_R \Lip_{V_L}(\T_h^d;
  M)$ then
  \begin{equation*}
    \P(|f(v)| > (t-1)M)\, \le\,
  \exp\left(-\frac{d(d-1)^{t-1}}{5(M+1)}\right).
  \end{equation*}
  for all vertices $v$ of $\T_h^d$ and integers $t \ge 1$.
\end{thm}
In Section~\ref{sec: sum}, we discuss extensions of the above theorem to any $d>2$ and $M\ge1$.

%
%

\subsection{Proof idea}\label{sec:proof_idea}

We now shortly discuss the general idea behind the proofs. Our goal
is to show that a certain set is small (the set of functions $f$ so that
$f(v)$ is far from $\phase(f)$). The basic tool we use is extremely
simple: By definition, a set $P$ is not larger than another set $Q$
if there exists a one-to-one map from $P$ to $Q$. To show that $P$
is {\em much} smaller than $Q$, we just need to find a few-to-many
map $T$ from $P$ to $Q$, as the following (trivial) lemma shows.

\begin{lemma}
  \label{lemma:double-counting}
  Let $P$ and $Q$ be finite sets.
  Let $T$ be a map on $P$ such that
  $T(p) \subseteq Q$ for all $p \in P$.
  If there are $\alpha, \beta > 0$ such that for each $p \in P$ and each $q \in Q$,
  \[
  |T(p)| \geq \beta \quad \text{and} \quad |\{p' \in P \colon q \in T(p')\}| \leq \alpha,
  \]
then $|P|/|Q| \leq \alpha/\beta$.
\end{lemma}
\begin{proof} We have $\beta |P| \leq \sum_{p \in P} |T(p)| =
    \sum_{q \in Q} |\{p' \in P \colon q \in T(p')\}| \leq \alpha
    |Q|.$
\end{proof}

The heart of the proof is constructing such a map $T$ and
establishing that it is few-to-many. Let us focus on the specific
case of grounded homomorphisms on trees since it is the simplest
case to explain and already captures the essence of our approach.
Let $G=\T_h^d$ be the complete $(d-1)$-ary tree of height $h\ge 1$
(in which all internal vertices have degree $d$). Recall that a
grounded homomorphism on $G$ is a function $f \colon V(G) \to \Z$ changing
by exactly one along every edge and having $f(v_0) = 0$ for every
leaf $v_0$. The goal is to estimate the probability that a random
grounded homomorphism takes a large value at some fixed vertex $v$
of $G$. More precisely, for some integer $t\ge 1$, we want to
estimate $\P(f(v)>t)$. By symmetry, this is the same as
$\P(f(v)<-t)$. An interesting case to keep in mind is the case that
$v$ is the root vertex. While it is possible for a
grounded homomorphism function to take a large value at the root, we
will show that this behavior is highly atypical.

%

We denote by $Q$ the set of all grounded homomorphisms on $G$ and by
$P$ the set $\{f \in Q \colon f(v) > t\}$. We wish to show that the size
of $P$ is much smaller than the size of $Q$ and we will do so by
constructing a few-to-many map $T$ from $P$ to $Q$ and invoking
Lemma~\ref{lemma:double-counting}. For fixed $f$ in $P$, we now
describe how to define $T(f)$.

Let $G^{\le 2}$ be the graph defined on the same vertex set as $G$
so that two vertices are connected in $G^{\le 2}$ if their distance
in $G$ is at most $2$. Consider $A = A(f)$, where
\begin{center}
$A(f)$ is the connected component of $v$ in the subgraph of $G^{\le
2}$\\ induced by the set of vertices $w$ having $f(w) \geq 2$.
\end{center}
Let $X=X(f)$ be the outer boundary of $A$ in $G$, i.e., the set of
vertices $w$ which are not in $A$ but are connected to $A$ by an edge.
Let $Y=Y(f)$ be the outer boundary of $A \cup X$ in $G$. So, we have
a set $A$ and two ``shells'' of it, $X$ and $Y$. Since $f$ changes
by exactly one along edges, we must have that
\begin{equation}\label{eq:hom_bdry_values}
  f(X) = \{1\} \quad \text{and} \quad f(Y) = \{0\}.
\end{equation}
Also, since $f(v)>t$ and $f$ changes by exactly one along edges, we
must have $B(v,t)\subseteq A\cup X$, which implies that
\begin{equation}\label{eq:X_bound}
  |X|\ge d(d-1)^{t-1}.
\end{equation}
From $f$ we can define many other homomorphisms: for every $s \in
\{-1,1\}^X$, we can define
$$h_s(w) = \left\{ \begin{array}{ll}
f(w) - 2 & w \in A, \\
s_w & w \in X, \\
f(w) & w \not \in A \cup X, \\
\end{array} \right.$$
and it follows easily from \eqref{eq:hom_bdry_values} that $h_s$ is
indeed a homomorphism. Thus we define
$$T(f) = \big\{h_s \colon s \in \{-1,1\}^X\big\}.$$

%

We now study the few-to-many property of $T$. It is convenient to
partition the set of functions $P$ according to the sizes of $A$ and
$X$. We let
\begin{equation*}
  P_{\alpha,\beta}:=\{f\in P\, \colon\, |A(f)|=\alpha\text{ and } |X(f)|
  = \beta\}.
\end{equation*}
We claim that $T$ is few-to-many from $P_{\alpha,\beta}$ to $Q$: On
the one hand,
\begin{equation}\label{eq:T_f_bound}
   |T(f)| = 2^{|X(f)|} = 2^{\beta}
\end{equation}
for all $f\in P_{\alpha,\beta}$. On the other hand, to recover $f$
from $h_s$ we just need to describe the set $A(f)$ (since given $A$
we know $X$, $f$ on $A$ is a translate of $h_s$, $f$ on $X$ is $1$,
and $f$ on the complement of $A \cup X$ equals $h_s$). Since $A$ is
connected in $G^{\le2}$, the number of possible sets $A$ of a given
size $\alpha$ is at most $d^{4\alpha}$ (see
Lemma~\ref{lemma:connected-sets} below). Thus, for every $h$ in $Q$,
\begin{equation}\label{eq:T_inverse_bound}
\big| \big\{ f\in P_{\alpha, \beta}\, :\, h \in T(f)\big\}\big| \le
d^{4\alpha}.
\end{equation}
Now, we use the fact that subsets of (the vertex set of) the tree which do not contain leaves
expand very well (see Claim~\ref{cl:tree_expansion} below): Since
$A(f)$ does not contain leaves by its definition and the fact that
$f$ equals zero on the leaves, we have
\begin{equation}\label{eq:tree_exp}
(d-2) |A(f)| < |X(f)|
\end{equation}
for all $f\in P$. Putting together \eqref{eq:T_f_bound},
\eqref{eq:T_inverse_bound} and \eqref{eq:tree_exp} and applying
Lemma~\ref{lemma:double-counting} gives
\begin{equation*}
  \frac{|P_{\alpha,\beta}|}{|Q|} \le \frac{d^{4\alpha}}{2^{\beta}}
  \le \frac{d^{4\beta / (d-2)}}{2^{\beta}} \le 2^{-\beta/2}
\end{equation*}
for large $d$. By summing over all possibilities for $\alpha$ and
$\beta$ and using \eqref{eq:X_bound} we may conclude that
$$\Pr (f(v) > t) = \frac{|P|}{|Q|}
\leq \exp \left(-\frac{1}{5}d(d-1)^{t-1}\right),$$ for all large
$d$, as we wanted to prove.

\subsection{Readers' guide}

Section~\ref{sec: prel} gives some preliminary results on graphs and
probability spaces. Sections~\ref{sec:Lipschitz} and~\ref{sec:Lipschitz-trees}
contain our investigation of the behavior of Lipschitz functions on good expanders
and on trees, respectfully. In Section~\ref{sec:hom}, we study the behavior 
of homomorphisms on good bi-expanders. Finally, Section~\ref{sec: sum}
consists of a short summary and suggestions for future research.

\section{Preliminaries}
\label{sec: prel}

We start with fixing some notation. In the following, $G$ is an
arbitrary graph. We recall that the graph $G^{\leq 2}$ is defined as follows: it
has the same vertex set as $G$ and $\{u,v\} \in E(G^{\leq 2})$ iff
$\dist_G(u,v) \leq 2$, where $\dist_G$ is the graph distance in $G$, that is,
the length of a shortest path connecting a pair of vertices in $G$.
For a set $A$ of vertices of $G$, define the {\em neighborhood} of
$A$ by
\[
\n(A) = \{ b \in V(G) \colon \text{there exists an $a \in A$ so that $\{a,b\} \in E(G)$}\}.
\]
Define the {\em outer boundary} of $A$ as
$$\b(A) = \n(A) \setminus A.$$
The {\em $2$-outer boundary} of $A$ is
$$\b^2(A) = \n(\n(A)) \setminus (A \cup \n(A)).$$
Denote balls in (the graph metric on) $G$ by
$$B(v,t) = \{ w \in V(G) \colon \dist_G(v,w) \leq t\}.$$
Finally, for a finite set $X$, denote by $\cP(X)$ the family of all subsets of $X$.

In the remainder of this section, we list and prove some preliminary, simple facts.
\subsection{Counting connected sub-graphs}

The following lemma uses breadth-first search to bound
the number of connected sets in a graph.

\begin{lemma}
  \label{lemma:connected-sets}
  Let $a$ be an integer, let $H$ be an arbitrary graph
  with maximum degree $\Delta$, and let $v \in V(H)$.
  The number of connected sets $A \subseteq V(H)$ such that $v \in A$ and $|A| = a$ does not exceed $\Delta^{2a-2}$.
\end{lemma}
\begin{proof}
  For every such $A$, we fix an arbitrary spanning tree $T_A$ of $A$. We perform a breadth-first search on $T_A$, starting and ending at $v$ and passing through every edge exactly twice. Since every spanning tree of $A$ has exactly $a-1$ edges, the number of possibilities for such a walk (and hence for $A$) is not larger than the number of walks of length $2a-2$ in $H$ that start at $v$.
\end{proof}

\subsection{Expanders}
\label{sec:expandres}

In this section, we establish some basic properties of $\lambda$-expanders. Some of them were already mentioned in our discussion in Section~\ref{sec:Lipschitz-functions}. Our proof of Lemma~\ref{lemma:phase} relies on the fact that every pair of sufficiently large vertex sets in a $\lambda$-expander is connected by an edge, which follows directly from Definition~\ref{defn:expander}.

\begin{prop}[connectivity]
  \label{prop:exp-small-sets}
  Let $G$ be a $d$-regular $n$-vertex $\lambda$-expander. Then for every two sets $A, B \subseteq V(G)$ satisfying $\min\{|A|, |B|\} > \frac{\lambda n}{d}$, there is an edge of $G$ joining $A$ and $B$.
\end{prop}
\begin{proof}
Since $\frac{d}{n}|A||B| > \lambda \sqrt{|A||B|}$, it follows from Definition~\ref{defn:expander} that $e(A,B) > 0$.
\end{proof}

We can now define the phase of a function and prove its properties.

\begin{proof}[{Proof of Lemma~\ref{lemma:phase}}]
  Let $G$ be a $d$-regular $n$-vertex $\lambda$-expander and let $v_0$ be an arbitrary vertex of $G$. Fix an $f \in \Lip_{v_0}(G;M)$ and let $k$ be the smallest integer such that
  \begin{equation}
    \label{eq:phase-f-disc}
    \big| \{v \in V(G) \colon f(v) \leq k\} \big| > \frac{\lambda n}{d}.
  \end{equation}
  Since $f$ is $M$-Lipschitz, there are no edges in $G$ between the sets $\{v \colon f(v) \leq k\}$ and $\{v \colon f(v) > k+M\}$. It follows from Proposition~\ref{prop:exp-small-sets} and~\eqref{eq:phase-f-disc} that the latter set has at most $\frac{\lambda n}{d}$ elements. Hence, by minimality of $k$,
  \[
  |\{v \colon f(v) \not\in \{k, \ldots, k+M\}\}| = |\{v \colon f(v) < k\}| + |\{v \colon f(v) > k+M\}| \leq \frac{2\lambda n}{d}.
  \]

This shows that the set of integers $k$
satisfying~\eqref{eq:phase_def_hom} is non-empty for all $f$. We now
describe a way to define $\phase(f)$ so that $\phase(-f)=-\phase(f)$
for $f$ which is not identically zero. Fix some total order on
$\Lip_{v_0}(G;M)$ (e.g., the lexicographic order). For every $f\in
\Lip_{v_0}(G;M)$ which is not identically zero, let $b(f)$ be the
larger in this total order between $f$ and $-f$, and let $s(f)$ be
the smaller of the two. Define $\phase(b(f))$ as the interval $\{k,
\ldots, k+M\}$ for the minimal $k$
satisfying~\eqref{eq:phase_def_hom} for $b(f)$, and define
$\phase(s(f)) = - \phase(b(f))$.
\end{proof}

The following standard proposition shows vertex expansion in expanders.

\begin{prop}[expansion]
  \label{prop:exp-N-growth}
  Let $G$ be a $d$-regular $n$-vertex $\lambda$-expander. Then for every $A \subseteq V(G)$,
  \[
  |\n(A)| \geq \min\left\{\frac{n}{2}, \frac{d^2}{4\lambda^2}|A| \right\}.
  \]
\end{prop}
\begin{proof}
  Let $A \subseteq V(G)$. We may assume that $A \neq \emptyset$ and $|\n(A)| < n/2$. Since there are no edges between $A$ and the complement of $\n(A)$, then by Definition~\ref{defn:expander},
  \[
  d|A| = e(A,\n(A)) \leq \frac{d}{n}|A||\n(A)| + \lambda\sqrt{|A||\n(A)|} \leq \frac{d|A|}{2} +
  \lambda\sqrt{|A||\n(A)|}.\qedhere
  \]
\end{proof}

Proposition~\ref{prop:exp-N-growth} has the following immediate corollary.

\begin{cor}[large boundary]
  \label{cor:exp-bdry-growth}
  Let $G$ be a $d$-regular $n$-vertex $\lambda$-expander. Then for every $A \subseteq V(G)$ with $|A| \leq \frac{n}{4}$,
  \[
  |\b(A)| \geq \min\left\{\frac{n}{4}, \left(\frac{d^2}{4\lambda^2} - 1 \right)|A| \right\}.
  \]
\end{cor}

The following proposition gives an estimate on the growth of balls in expanders.

\begin{prop}[volume growth]
  \label{prop:exp-balls-growth}
  Let $G$ be a $d$-regular $n$-vertex $\lambda$-expander. Then for every non-negative integer $t$ and every $v \in V(G)$,
  \begin{equation}
    \label{eq:Bvt}
    |B(v,t)| \geq \min\left\{ \frac{n}{2} , \left(\frac{d}{2\lambda}\right)^{2t} \right\}.
  \end{equation}
\end{prop}
\begin{proof}
  Fix $v$ and prove~\eqref{eq:Bvt} by induction on $t$.
  The bound trivially holds for $t = 0$.
  Assume that \eqref{eq:Bvt} holds for $t \geq 0$.
  Since $B(v,t+1) \supseteq \n(B(v,t))$, Proposition~\ref{prop:exp-N-growth} implies that
  \[
  |B(v,t+1)| \geq \min\left\{\frac{n}{2}, \frac{d^2}{4\lambda^2}|B(v,t)|
  \right\}.\qedhere
  \]
\end{proof}

Proposition~\ref{prop:exp-balls-growth} implies the following bound on the diameter of expanders.

\begin{cor}[diameter]
  \label{cor:exp-diam}
  Let $G$ be a $d$-regular $n$-vertex $\lambda$-expander. If $\lambda < d/2$, then
  \[
  \diam(G) \leq \left(\log \frac{d}{2\lambda} \right)^{-1} \cdot \log n.
  \]
\end{cor}
\begin{proof}
  Let $v$ and $w$ be two arbitrary vertices of $G$ and let $t$ be the minimum integer such that both $B(v,t)$ and $B(w,t)$ contain more than $\lambda n/d$ vertices. Since $\lambda n/d < n/2$, by Proposition~\ref{prop:exp-balls-growth},
  \[
  t \leq \left(2 \log \frac{d}{2\lambda}\right)^{-1}\cdot \log\left(\frac{\lambda n}{d}\right)\le \left(2 \log \frac{d}{2\lambda}\right)^{-1} \cdot \log n - \frac{1}{2}.
  \]
  By Proposition~\ref{prop:exp-small-sets}, $B(v,t)$ and $B(w,t)$ are joined by an edge and hence $\dist_G(v,w) \leq 2t+1$.
\end{proof}

\subsection{Bi-expanders}
\label{sec:bi-expandres}

In this section, we present some basic properties of $\lambda$-bi-expanders. As most of these properties are natural bipartite analogues of the statements presented in Section~\ref{sec:expandres} (and given the obvious similarity between Definitions~\ref{defn:expander} and \ref{defn:bi-expander}), we leave most of the proofs as an exercise for the reader.

\begin{prop}[connectivity]
  \label{prop:bi-exp-small-sets}
  Let $G$ be a $d$-regular $2n$-vertex $\lambda$-bi-expander and let $V_0$ and $V_1$ be the two color classes of $G$. Then for every two sets $A \subseteq V_0$ and $B \subseteq V_1$ satisfying $\min\{|A|, |B|\} > \frac{\lambda n}{d}$, there is an edge of $G$ joining $A$ and $B$.
\end{prop}

We can now define the phase of a homomorphism function on
a~$\lambda$-bi-expander and prove its properties.

\begin{proof}[Proof of Lemma~\ref{lemma:phase_hom}]
  Let $G$ be a $d$-regular $2n$-vertex $\lambda$-bi-expander, let $V_0$ and $V_1$ be the two color classes of $G$, and let $v_0$ be an arbitrary vertex of $G$. Fix $f \in \Hom_{v_0}(G)$, let $k$ be the smallest integer such that
  \begin{equation}
    \label{eq:phase-f-disc-bi}
    |\{v \in V_{1-i} \colon f(v) < k\}| > \frac{\lambda n}{d}
  \end{equation}
  for some $i \in \{0,1\}$ and fix (the unique) such $i$. Since $f$ is a homomorphism, there are no edges in $G$ between the sets $\{v \in V_{1-i} \colon f(v) < k\}$ and $\{v \in V_i \colon f(v) > k\}$. It follows from Proposition~\ref{prop:bi-exp-small-sets} and~\eqref{eq:phase-f-disc-bi} that the latter set has at most $\frac{\lambda n}{d}$ elements. Since all values taken by $f$ on $V_i$ have the same parity, it follows from the minimality of $k$ that $\{v \in V_i \colon f(v) < k\} = \{v \in V_i \colon f(v) < k-1\}$ and hence
  \[
  |\{v \in V_i \colon f(v) \neq k\}| = |\{v \in V_i \colon f(v) < k - 1\}| + |\{v \in V_i \colon f(v) > k\}| \leq \frac{2\lambda n}{d}.
  \]
  Finally, suppose that $\lambda < d/3$.
  Let $k = \phase(f)$, and let $i = i^*$ be such that~\eqref{eq:phase_def_hom} holds. Since $f$ is a homomorphism, there are no edges in $G$ between the sets $\{v \in V_i \colon f(v) = k\}$ and $\{v \in V_{1-i} \colon |f(v) - k| \geq 2\}$. By the definition of $\phase$ and the assumption on $\lambda$, it follows that the former set has more than $\frac{\lambda n}{d}$ elements and hence, by Proposition~\ref{prop:bi-exp-small-sets}, the latter set has at most $\frac{\lambda n}{d}$ elements. Therefore,
  \[
  |\{v \in V(G) \colon |f(v) - k| \geq 2\}| = |\{v \in V_i \colon f(v) \neq k\}| + |\{v \in V_{1-i} \colon |f(v) - k| \geq 2\}| \leq \frac{3\lambda n}{d}.
  \qedhere
  \]
\end{proof}

\begin{prop}[expansion]
  \label{prop:bi-exp-N-growth}
  Let $G$ be a $d$-regular $2n$-vertex $\lambda$-bi-expander. Then for every $A \subseteq V(G)$,
  \[
  |\n(A)| \geq \min\left\{\frac{n}{2}, \frac{d^2}{4\lambda^2}|A| \right\}.
  \]
\end{prop}

\begin{cor}[large boundary]
  \label{cor:bi-exp-bdry-growth}
  Let $G$ be a $d$-regular $2n$-vertex $\lambda$-bi-expander. Then for every $A \subseteq V(G)$ with $|A| \leq \frac{n}{4}$,
  \[
  |\b(A)| \geq \min\left\{\frac{n}{4}, \left(\frac{d^2}{4\lambda^2} - 1 \right)|A| \right\}.
  \]
\end{cor}

\begin{prop}[volume growth]
  \label{prop:bi-exp-balls-growth}
  Let $G$ be a $d$-regular $2n$-vertex $\lambda$-bi-expander. Then for every non-negative integer $t$ and every $v \in V(G)$,
  \[
  |B(v,t)| \geq \min\left\{ \frac{n}{2} , \left(\frac{d}{2\lambda}\right)^{2t} \right\}.
  \]
\end{prop}

\begin{cor}[diameter]
  \label{cor:bi-exp-diam}
  Let $G$ be a $d$-regular $2n$-vertex $\lambda$-bi-expander. If $\lambda \leq d/8$, then
  \[
  \diam(G) \leq \left(\log \frac{d}{2\lambda} \right)^{-1} \cdot \log n + 1.
  \]
\end{cor}

\section{Lipschitz functions}

\label{sec:Lipschitz}

Assume that $G$ is a connected $d$-regular $n$-vertex
$\lambda$-expander that is $M$-good, which roughly means that
$\lambda \ll d/(M\log d)$. Let $v$ and $v_0$ be two (not necessarily
distinct) vertices of $G$ and let $t$ be a positive integer. Recall
the definition of $\phase$ from Lemma~\ref{lemma:phase}. We will
estimate the probability that a uniformly chosen random function $f
\in_R \Lip_{v_0}(G; M)$ is in the event
\[
\Omega = \big\{f \in \Lip_{v_0}(G;M) \colon \dist(f(v), \phase(f)) > (t-1)M \big\}.
\]
We first briefly describe our strategy. Our proof of Theorem~\ref{thm:phase_dev} is divided into two  (independent) parts.

In the first part, described in Sections~\ref{sec:constructing-map-T} and~\ref{sec:properties-T}, we construct a map $T \colon \Omega \to \cP(\Lip_{v_0}(G;M))$ such that the set $T(f)$ is large for every $f \in \Omega$. Moreover, for every $g$,
we bound the size of the set $\{f \in \Omega \colon g \in T(f) \}$.
This is crucial in estimating the probability of $\Omega$ using Lemma~\ref{lemma:double-counting}. In this part of the proof, we do not use the assumption that $G$ is a $\lambda$-expander.

In the second part of the proof, Section~\ref{sec:bound-omega}, we derive the claimed bound on $\P(\Omega)$ using the properties of the transformation $T$ and the underlying graph $G$. Here, we strongly use the assumption that $G$ is a good expander.

In fact, we partition $\Omega$ into two parts and argue as above on
each part:
\[
\Omega^+ = \{f \in \Omega \colon f(v) > \max \phase(f) \}
\quad \text{and} \quad
\Omega^- = \{f \in \Omega \colon f(v) < \min \phase(f)  \}.
\]
Since $\phase(- f) = - \phase(f)$ by the definition of the phase, the map $f \mapsto -f$ is a
bijection between $\Omega^+$ and $\Omega^-$. Thus,
\begin{equation}\label{eq:P_Omega_bound}
\P (\Omega) \leq \P(\Omega^+) + \P(\Omega^-) \leq 2 \P(\Omega^+) .
\end{equation}
Hence, from now on we can focus our attention on the event
$\Omega^+$.

\subsection{Constructing the transformation $T$}

\label{sec:constructing-map-T}

Let $f \in \Omega^+$. Observe that $f$ is not the zero function. In
other words, let $f \in \Lip_{v_0}(G;M)$ satisfy $f(v) > k + tM$,
where $k$ is the unique integer such that $\{k, \ldots, k+M\} =
\phase(f)$. To make our argument more general, for the remainder of
this section and in Section~\ref{sec:properties-T}, let us disregard
the fact that $G$ is an $M$-good expander and the precise definition
of $k$. Let us only assume that $G$ is an arbitrary finite connected
graph with two (not necessarily distinct) fixed vertices $v_0$ and
$v$, that an arbitrary function $k \colon \Lip_{v_0}(G;M)\to\Z$ is given
and that $\Omega^+$ is a subset of all $f \in \Lip_{v_0}(G;M)$
satisfying $f(v)
> k+M$ for $k = k(f)$. Our only requirement on the function $k$ is that for each
$f\in\Lip_{v_0}(G;M)$ there exists some vertex $w$ on which $f(w)\le
k+M$ for $k = k(f)$. Let
\begin{center}
$A(f)$ be the connected component of the vertex $v$ in the subgraph
of $G^{\leq 2}$ \\ induced by the set of vertices  $\{w \colon
f(w) > k+M\}$.
\end{center}
Our requirement on $k$ implies that $A(f)$ does not contain all vertices of
$G$. Let us further partition the event $\Omega^+$. Let $\C_v$ be
the family of all sets of vertices $A \subseteq V(G)$ such that $v
\in A$ and $A$ is connected in $G^{\leq 2}$. For every $A \in \C_v$,
let
\[
\Omega_A^+ = \{f \in \Omega^+ \colon A(f) = A\}.
\]

\begin{claim}\label{cl:f_on_A}
Let $A \in \C_v$ and $f \in \Omega_A^+$.
Set $X = \b(A)$.
Then, the following properties hold:
\begin{enumerate}
\item \label{item:fA}
 $\min f(A) > k+M$,
\item \label{item:fX}
  $ f(X) \subseteq \{k+1,k+2,\ldots,k+M\}$, and
\item \label{item:fd2A}
  $\max f(\b^2(A)) \leq k+M$.
\end{enumerate}
\end{claim}

\begin{proof}
\begin{enumerate}
\item Follows since $A$ is a subset of the set of vertices $w$
such that $f(w) > k+M$.
\item  Since $A$ is defined as a connected component,
if $w$ in $X$ satisfies $f(w) > k+M$,
then $w$ is in $A$ as well, which is a contradiction ($X \cap A = \emptyset$).
Since $f$ is $M$-Lipschitz, and $\min f(A) > k+M$,
we have $\min f(X) \geq \min f(\n(A)) > k$.
 \item Similarly to \ref{item:fX}., if $w\in\b^2(A)$ and satisfies
 $f(w)>k+M$, then, since $A$ is defined as a connected component in
 $G^{\leq 2}$, we have $w\in A$, a contradiction.\qedhere
\end{enumerate}

\end{proof}

For the discussion, fix an $f \in \Omega^+_A$ for some $A \in \C_v$ and set $X = \b(A)$.
Here is a first hint on how to define $T(f)$. For any $s \in \Z^X$, define
$h_s = h_s(f)$, a map from $V(G)$ to $\Z$, by
\[
h_s(w) =
\begin{cases}
  k+M & w \in A ,\\
  k+s_w & w \in X, \\
  f(w) & w \not\in A \cup X.
\end{cases}
\]
Clearly, not every $s \in \Z^X$ gives rise to an $M$-Lipschitz function $h_s$. Still, it is quite easy to identify a large subset of $\Z^X$ that does have this property. To this end, for every $x \in X$, let
$$u_x = u_x(f) = \min\left(\{f(w) + M - k \colon w \in N(x) , w \notin A \cup X \} \cup \{M\}\right)
$$
and
$$\ell_x = \ell_x(f) = \max\{f(w) - M - k \colon w \in N(x) \cap A\}.$$

The following propositions clarifies the relation between the sequences defined above.
\begin{prop}
  \label{prop:lu}
  Let $A \in \C_v$, let $f \in \Omega_A^+$, and let $X = \b(A)$. Then, for every $x \in X$,
  \begin{equation}\label{eq:lu}
  1 \leq \ell_x(f) \leq f(x) - k \leq u_x(f) \leq M.
  \end{equation}
\end{prop}

The proposition shows that $u_x$ is an {\em upper} bound on the value
of $f(x)-k$, and $\ell_x$ is a {\em lower} bound.

\begin{proof}
  Fix $x \in X$.  There are four inequalities to prove.
  First, let $w_x$ be an element of $N(x) \cap A$ such that $\ell_x(f) = f(w_x) - M - k$.
  Since $w_x \in A$, it follows from~\ref{item:fA} in Claim~\ref{cl:f_on_A} that $f(w_x) > k+M$, which proves $\ell_x(f) \geq 1$.
  Second, since $\{x,w_x\}$ is an edge of $G$ and $f$ is $M$-Lipschitz, it follows that $f(x) \geq f(w_x) - M = \ell_x(f) + k$.
  Third, because $f(x) \leq k+M$ (see~\ref{item:fX} in Claim~\ref{cl:f_on_A}) and $f(x) \leq f(w) + M$ for every $w \in N(x)$,
  \[
  f(x) - k \leq \min\left(\{f(w) + M - k \colon w \in N(x) ,w \notin  A \cup X \} \cup \{M\}\right) = u_x(f).
  \]
  Finally, the inequality $u_x(f) \leq M$ follows directly from the definition of $u_x(f)$.
\end{proof}

Using the sequence $(u_x)$, we can define a large family of Lipschitz functions. Let
\[
S = S(f) = \big\{s \in \Z^X \colon s_x \in \{0, \ldots, u_x(f)\} \text{ for every $x \in X$} \big\}.
\]

\begin{claim}
  \label{claim:S}
  For every $f$ in $\Omega^+$
  and for every $s \in S(F)$,
  the map $h_s = h_s(f)$ is an $M$-Lipschitz function
(but there is no guarantee that $h_s(v_0)=0$.)
\end{claim}

\begin{proof}
  Fix  $s \in S$ and let $\{w_1, w_2\}$ be an edge of $G$. It suffices to show that $|h_s(w_1) - h_s(w_2)| \leq M$. In order to prove this inequality, we consider several cases depending on the locations of the vertices $w_1$ and $w_2$.

  \medskip
  \noindent
  {\bf Case 1.} If $w_1, w_2 \in A$, then $|h_s(w_1) - h_s(w_2)| = 0$.

  \medskip
  \noindent
  {\bf Case 2.} If $w_1, w_2 \in X$, then $|h_s(w_1) - h_s(w_2)| = |s_{w_1} - s_{w_2}| \leq \max\{u_{w_1}, u_{w_2}\} \leq M$.

  \medskip
  \noindent
  {\bf Case 3.} If $w_1, w_2 \not\in A \cup X$, then $|h_s(w_1) - h_s(w_2)| = |f(w_1) - f(w_2)| \leq M$.

  \medskip
  \noindent
  {\bf Case 4.} If $w_1 \in A$ and $w_2 \in X$, then $|h_s(w_1) - h_s(w_2)| = |k+M - k - s_{w_2}| \leq M$.

  \medskip
  \noindent
  {\bf Case 5.} If $w_1 \in X$ and $w_2 \not\in A \cup X$, then,
  by definition of $u_{w_1}$,
  \[
  h_s(w_1) - h_s(w_2) = k + s_{w_1} - f(w_2) \leq k + u_{w_1}(f) - f(w_2)
  \leq M .
  \]
 On the other hand, by~\ref{item:fd2A} in Claim~\ref{cl:f_on_A},
  \[
  h_s(w_1) - h_s(w_2) = k + s_{w_1} - f(w_2) \geq k - k - M = -M.
  \]

  \medskip
  Since there are no edges between $A$ and $(A \cup X)^c$, the proof is now complete.
\end{proof}

One might be tempted to suggest $T(f) = \{h_s \colon s \in S(f)\}$.
Although this is a reasonable guess, one needs to be somewhat careful.
As $v_0$ is arbitrary, it might happen that 
$h_s(v_0) \neq 0$ for some $s \in S$. Luckily, the following claim reassures us that this is not a serious issue.
For this, we define the ``shift'' operator $P_{v_0}$
on functions $h: V(G) \to \Z$ by
$$P_{v_0}(h) = h - h(v_0).$$
Clearly, $P_{v_0}(h)(v_0) = 0$
and if $h$ is $M$-Lipschitz, so is $P_{v_0}(h)$.

\begin{claim}
  \label{claim:one-to-one}
Let $f\in\Omega^+$ and $S = S(f)$. The shift $P_{v_0}$ is one-to-one
on $\{h_s \colon s \in S\}$.
\end{claim}
\begin{proof}
  Let $s \in S$.
  Since $h_s(v) = k+M$ and $h_s - P_{v_0}(h_s)$ is a constant function,
  \[
  h_s = P_{v_0}(h_s) + (k+M - P_{v_0}(h_s)(v)).\qedhere
  \]
\end{proof}

Finally, define $T \colon \Omega^+ \to \cP(\Lip_{v_0}(G;M))$ by
\begin{equation}
  \label{eq:Tf-def}
  T(f) =  \left\{P_{v_0}(h_s):  s \in S(f)\right\}.
\end{equation}

\subsection{Properties of $T$}

\label{sec:properties-T}

Before we establish the key properties of $T$, we need to further refine our partition of $\Omega^+$. For every $A \in \C_v$, let
\[
\S(A) = \left\{S(f) \colon f \in \Omega_A^+\right\}
\]
and for every $S \in \S(A)$, let
\[
\Omega_{A,S}^+ = \left\{ f \in \Omega_A^+ \colon S(f) = S \right\}.
\]
We are now ready to prove a key lemma.

\begin{lemma}
  \label{lemma:Lipschitz-main}
Let $A \in \C_v$ and $X = \b(A)$.
For every $S = \prod_{x \in X} \{0,\ldots,u_x\}$ in $\S(A)$, the following holds:
  \begin{enumerate}
  \item \label{item:Lipschitz-main-1}
    If $f \in \Omega_{A,S}^+$, then $|T(f)| = |S|$.
  \item \label{item:Lipschitz-main-2}
    For every $h \in \Lip_{v_0}(G;M)$, we have
    \[
    |\{f \in \Omega_{A,S}^+ \colon h \in T(f)\}| \leq M (2|A| + 1)(2M+1)^{|A|} \big|S_-\big|,
    \]
    where
    \[
S_-  = \prod_{x \in X} \{1,\ldots,u_x\}.
\]
  \end{enumerate}
\end{lemma}
\begin{proof}
  {\it \ref{item:Lipschitz-main-1}.} By \eqref{eq:Tf-def} and Claim~\ref{claim:one-to-one},
  \[
  |T(f)| = |\{P_{v_0}(h_s) \colon s \in S\} | = |S|.
  \]

{\it \ref{item:Lipschitz-main-2}.}
Fix some $h \in \Lip_{v_0}(G;M)$.
Assume that we are given an integer $k$ and some function $f_0 \colon A \cup X \to \Z$. We claim that there is at most one $f \in \Omega_{A,S}^+$ such that
\begin{itemize}
\item $h \in T(f)$,
\item $k = k(f)$, and
\item $f(w) = f_0(w)$ for all $w \in A \cup X$.
\end{itemize}
It suffices to check that if there exists such an $f$, then we can
uniquely reconstruct its values on all vertices $w \notin A \cup X$
using only the fact that $h\in T(f)$ and the data in $h, f_0$ and
$k$. Recall that $h = P_{v_0}(h_s)$ for some $s \in S$ and that
$f(w) = h_s(w)$ for all $w \not\in A \cup X$. Hence, it suffices to
reconstruct $h_s(v_0)$. But $h_s(v_0) = h_s(v) - h(v) = k + M -
h(v)$.

Therefore, in order to bound the number of possible $f$s with $h\in
T(f)$, it suffices to bound the number of pairs $(k, f_0)$ for which
there exists an $f$ as above.

First, we bound the number of possibilities for $k$. If $v_0 \notin
A \cup X$ then since $h_s(v_0) = f(v_0) = 0$, we have $h(v)=k+M$ and
hence $k$ is uniquely determined by $h$. If $v_0\in A\cup X$,
observe that since any $f\in\Omega_{A,S}^+$ satisfies $f(v_0)=0$ and
can change by at most $M$ along each edge of $G$, there exists some
$x_0\in X$ such that $|f(x_0)|\le M|A|$ for all
$f\in\Omega_{A,S}^+$. It follows from part~\ref{item:fX}.\ of
Claim~\ref{cl:f_on_A} that $k\in[-M|A|-M,M|A|-1]$ and hence there
are at most $M(2|A|+1)$ options for $k(f)$ for $f\in\Omega_{A,S}^+$.


Now, fix $k$ and bound the number of possibilities for $f_0$. Start by bounding the number of possible values of $f_0$ on $X$.
Let $f \in \Omega_{A,S}^+$ be such that $h \in T(f)$. Since $S = S(f)$ by the definition of $\Omega_{A,S}^+$, by Proposition~\ref{prop:lu}, $1 \leq f(x) - k \leq u_x$ for every $x \in X$. In other words, $(f_0(x)-k)_{x \in X} \in S_-$.

Finally, we bound the number of possible values of $f_0$ on $A$. Fix
some $s \in S_-$ and suppose that $f_0(x) = k+s_x$ for every $x \in
X$. We bound the number of ways we can extend an $M$-Lipschitz $f_0$
from $X$ to $A \cup X$. Let $A'\subseteq A$ be a connected component
of $A$ in $G$. Observe that to specify $f_0$ on $A'$ it suffices to
fix a spanning tree of $A'$ in $G$ and specify the difference of
values of $f_0$ on the edges of this spanning tree and on a single
edge leading from $A'$ to $X$. Since the spanning tree has exactly
$|A'|-1$ edges, there are at most $(2M+1)^{|A'|}$ possibilities to
extend $f_0$ to $A'$. Multiplying this quantity over all connected
components of $A$ in $G$, we see that there are at most
$(2M+1)^{|A|}$ ways to extend $f_0$ from $X$ to $A$.
\end{proof}

Lemma~\ref{lemma:Lipschitz-main} already allows us to prove an
estimate on $|\Omega_{A,S}^+|$ using
Lemma~\ref{lemma:double-counting}.

\begin{cor}
  \label{cor:Omega_A,S}
  For every $A \in \C_v$ and $S \in \S(A)$,
  \[
  \frac{\big|\Omega_{A,S}^+\big|}{\big| T(\Omega_{A,S}^+)\big|} \leq
  M (2|A| + 1)(2M+1)^{|A|}\left(\frac{M}{M+1}\right)^{|\b(A)|}.
  \]
\end{cor}
\begin{proof}
  The claimed estimate follows directly from Lemmas~\ref{lemma:double-counting} and~\ref{lemma:Lipschitz-main} and the fact that
  \[
  \frac{|S_-|}{|S|} = \prod_{x \in X} \frac{u_x}{u_x+1} \leq \left(\frac{M}{M+1}\right)^{|X|}  \]
  with $X = \b(A)$.
\end{proof}

To derive a bound on $|\Omega_A^+|$, we
use the following property of the transformation $T$.

\begin{claim}
  \label{claim:TOmega_A,S-disjoint}
  For every $A \in \C_v$ and every $S \neq S'$ in $\S(A)$,
  \[
  T(\Omega_{A,S}^+) \cap T(\Omega_{A,S'}^+) = \emptyset.
  \]
\end{claim}

\begin{proof}
  Fix some $A \in \C_v$ and let $X = \b(A)$. Let $f \in \Omega_A^+$ and fix an arbitrary $h \in T(f)$.
  To prove the claim, we show that we can reconstruct $(u_x(f))_{x \in X}$ from $h$.
  Recall that $h_s(w) = f(w)$ for every $w \not\in A \cup X$.
  Hence, for $w \not\in A \cup X$,
  \[
  f(w) + M - k = h_s(w) - h_s(v) + 2M = h(w) - h(v) + 2M,
  \]
  where the first equality follows from the fact that $h_s(v) = k+M$,
  and the second equality follows from the fact that $h$ is a ``shift'' of $h_s$.
  Therefore, for every $x \in X$,
  \begin{align*}
    u_x(f) & = \min\left(\{f(w) + M - k \colon w \in N(x) , w \notin A \cup X \} \cup \{M\}\right) \\
    & = \min\left(\{h(w) - h(v) + 2M \colon w \in N(x), w \notin A \cup X \} \cup \{M\}\right).
    \qedhere
  \end{align*}
\end{proof}

Concluding, Corollary~\ref{cor:Omega_A,S} and Claim~\ref{claim:TOmega_A,S-disjoint} imply the following bound on the probability of $\Omega_A^+$.

\begin{cor}
  \label{cor:POmegaA+}
  For every $A \in \C_v$,
  \[
  \P(\Omega_A^+) \leq M(2|A| + 1)(2M+1)^{|A|}\left(\frac{M}{M+1}\right)^{|\b(A)|}.
  \]
\end{cor}

\subsection{Bounding $\P(\Omega)$}

\label{sec:bound-omega}

\begin{proof}[Proof of Theorem~\ref{thm:phase_dev}]
Let us again assume that $G$ is a connected $d$-regular $n$-vertex
$\lambda$-expander that is $M$-good and recall that $\Omega^+$ is
the family of $f \in \Lip_{v_0}(G;M)$ that satisfy $f(v) > k + tM$,
where $k = k(f) = \min\phase(f)$. Fix some $f \in \Omega^+$, recall
the definition of $A(f)$ from Section~\ref{sec:constructing-map-T}
and observe that by \eqref{eq:phase_def} in Lemma~\ref{lemma:phase},
\[
|A(f)| \leq |\{w \colon f(w) > k + M\}| \leq |\{w \colon f(w) \not\in \phase(f) \}| \leq \frac{2\lambda n}{d}.
\]
Moreover, since $f(v) > k + tM$ and $f$ is $M$-Lipschitz, it follows that $B_G(v,t-1) \subseteq A(f)$. Let
\[
\A_v = \left\{A \in \C_v \colon \text{$B(v,t-1) \subseteq A$ and $|A| \leq \frac{2\lambda n}{d}$}\right\}.
\]
Clearly, the sets $\Omega_A^+$ defined in Section~\ref{sec:constructing-map-T}, with $A \in \A_v$, form a partition of $\Omega^+$.

We are now ready to derive a bound on $\P(\Omega)$. Essentially, the
bound follows from Corollary~\ref{cor:POmegaA+} by showing that the
expansion properties of $G$ imply that $|\b(A)|$ is much larger than
$|A|$ for $A \in \A_v$. For an integer $\alpha$, let
\[
\A_{v,\alpha} = \{A \in \A_v \colon |A| = \alpha \}.
\]
Since every $A \in \A_v$ is connected in $G^{\leq 2}$ and the maximum degree of $G^{\leq 2}$ is at most $d^2$, then by Lemma~\ref{lemma:connected-sets},
\begin{equation}
  \label{eq:Ava-count}
  |\A_{v,\alpha}| \leq d^{4\alpha}.
\end{equation}
Let $\alpha$ be such that $\A_{v,\alpha}$ is non-empty.
In particular, $|B(v,t-1)| \leq \alpha \leq \frac{2\lambda n}{d}$.
Let
\[
b_\alpha = \min\left\{|\b(A)| \colon A \in \A_{v,\alpha} \right\} .
\]
By Corollary~\ref{cor:exp-bdry-growth},  since $\lambda \leq
\frac{d}{8}$ and hence $\alpha\le \frac{n}{4}$,
\begin{equation}
  \label{eq:partial-a}
  b_\alpha \geq
  \min\left\{\frac{n}{4}, \frac{d^2}{5\lambda^2} \alpha \right\} \geq \frac{d \alpha}{8 \lambda} ,
\end{equation}
where the last inequality holds as $\alpha \leq 2\lambda n/d$. Hence, by Corollary~\ref{cor:POmegaA+} and~\eqref{eq:Ava-count},
\begin{align}
  \label{eq:P-Ava}
  \sum_{A \in \A_{v,\alpha}} \P(\Omega_A^+) &
  \leq d^{4\alpha} \cdot M (2\alpha+1)(2M+1)^{\alpha}\left(\frac{M}{M+1}\right)^{b_\alpha} \\
  &\le d^{4\alpha} M^\alpha e^{2\alpha} (3M)^\alpha
  \exp\left(-\frac{b_\alpha}{M+1}\right)\le
  \exp\left(-\frac{b_\alpha}{M+1}+2\log(9Md^2)\alpha\right)\\
  \nonumber
  & \leq \exp \left( - b_\alpha \left(  \frac{1}{M+1} -  \frac{16 \lambda \log (9Md^2)}{d} \right) \right) \leq  \exp \left( - \frac{b_\alpha}{2(M+1)} \right),
\end{align}
where the last two inequalities follow from~\eqref{eq:partial-a} and
our assumption that $\lambda \leq \frac{d}{32(M+1)\log(9Md^2)}$,
respectively. Let us adopt the convention that when we write
$\sum_\alpha$ we mean a sum over $\{\alpha \colon
|\A_{v,\alpha}|\neq 0\}$. Since $A \cup \b(A) \supseteq B(v,t)$ for
every $A \in \A_v$, then $|B(v,t)| \geq \alpha + b_\alpha$ for every $\alpha$
as above and hence \eqref{eq:P_Omega_bound} and \eqref{eq:P-Ava}
give
\begin{align*}
  \P(\Omega) &\leq 2  \P(\Omega^+) = 2\sum_{A \in \A_v}
  \P(\Omega_A^+) = 2\sum_{\alpha} \sum_{A \in \A_{v,\alpha}}
  \P(\Omega_A^+)\le 2\sum_\alpha \exp \left( - \frac{b_\alpha}{2(M+1)}
  \right)\\
  & = 2\sum_\alpha \exp \left( - \frac{b_\alpha+\alpha-\alpha}{2(M+1)}
  \right)\le 2\exp \left( - \frac{|B(v,t)|}{5(M+1)}\right)\sum_\alpha \exp \left( - \frac{3b_\alpha}{10(M+1)} +
  \frac{\alpha}{5(M+1)}\right),
\end{align*}
Thus we need only estimate the last sum. Applying
\eqref{eq:partial-a} and our assumption that $\lambda \leq
\frac{d}{32(M+1)\log(9Md^2)}$ we have
\begin{align*}
  \sum_\alpha \exp \left( - \frac{3b_\alpha}{10(M+1)} +
  \frac{\alpha}{5(M+1)}\right)&\le
  \sum_\alpha\exp\left(-\left(\frac{3d}{80\lambda(M+1)}-\frac{1}{5(M+1)}\right)\alpha\right)\\
  &\le \sum_{\alpha}\exp\left(-\left(\frac{96\log(9Md^2)}{80}-\frac{1}{10}\right)\alpha\right)
  \le \sum_\alpha \exp(-2\alpha)\le \frac{1}{2},
\end{align*}
as required.
\end{proof}

\section{Lipschitz functions on trees}

\label{sec:Lipschitz-trees}

In this section, we prove Theorem~\ref{thm:grounded}. Recall the
definitions of $\T_h^d$, $v_r$, $V_L$ and $\Lip_{V_L}(\T_h^d;M)$
from Section~\ref{sec:grounded}. Fix integers $M\ge 1$, $d\ge
40(M+1)\log(M+1)$, $t\ge 1$ and a non-leaf vertex $v$ of $\T_h^d$
(since the theorem is trivial for leaf vertices). We will estimate
the probability that a uniformly chosen random function $f \in_R
\Lip_{V_L}(\T_h^d; M)$ is in the event
\[
\Omega = \big\{f \in \Lip_{V_L}(\T_h^d;M) \colon |f(v)|
> (t-1)M \big\}.
\]
By symmetry, $\P(\Omega)=2\P(\Omega^+)$ where
\[
\Omega^+ = \{f \in \Lip_{V_L}(\T_h^d;M) \colon f(v) > (t-1)M \}.
\]
Hence it suffices to bound the probability of $\Omega^+$. It is
convenient to introduce an auxiliary graph $\tilde{\T}_h^d$ by
taking the graph $\T_h^d$ and gluing the set of leaves $V_L$ to one
new vertex $v_0$. It is clear from the definition that the
probability distribution of $f(v)$ is the same when $f \in_R
\Lip_{V_L}(\T_h^d; M)$ and when $f \in_R \Lip_{V_L}(\tilde{\T}_h^d;
M)$. Hence we may focus on bounding the probability of the event
\[
\tilde{\Omega}^+ = \{f \in \Lip_{V_L}(\tilde{\T}_h^d;M) \colon f(v) > (t-1)M
\}.
\]
We may now use the results of Sections~\ref{sec:constructing-map-T}
and \ref{sec:properties-T} to the graph $\tilde{\T}_h^d$ with the
function $k(f)\equiv 0$ defined on $\Lip_{v_0}(\tilde{\T}_h^d)$. In
particular, defining $\C_v$ and
\[
\tilde{\Omega}_A^+ = \{f \in \tilde{\Omega}^+ \colon A(f) = A\}.
\]
as in Section~\ref{sec:constructing-map-T}, we deduce from
Corollary~\ref{cor:POmegaA+} that for every $A \in \C_v$,
\begin{equation}\label{eq:tilde_omega_+_bound}
\P(\tilde{\Omega}_A^+) \leq M(2|A| +
1)(2M+1)^{|A|}\left(\frac{M}{M+1}\right)^{|\b(A)|}.
\end{equation}
As in Section~\ref{sec:bound-omega} we again have that for every
$f\in\tilde{\Omega}^+$, $B_G(v,t-1) \subseteq A(f)$. In addition,
denoting $A=A(f)$, we have $(A\cup\b(A))\cap V_L=\emptyset$ by
Claim~\ref{cl:f_on_A}, since $f(V_L)=\{0\}$. Thus, letting
\[
\A_v = \left\{A \in \C_v \colon \text{$B(v,t-1) \subseteq A$ and
$(A\cup\b(A))\cap V_L=\emptyset$}\right\},
\]
the sets $\tilde{\Omega}_A^+$ with $A \in \A_v$ form a
partition of $\tilde{\Omega}^+$. It remains to use
\eqref{eq:tilde_omega_+_bound} to bound the probability of
$\tilde{\Omega}^+$. For this we will need that subsets of the tree
which do not contain leaves have large vertex expansion.
\begin{claim}\label{cl:tree_expansion}
  If $A$ is a non-empty subset of vertices of $\T_h^d$ which does
  not contain any leaves then
  \begin{equation*}
    |\b(A)|> (d-2)|A|.
  \end{equation*}
\end{claim}
\begin{proof}
  By induction on $|A|$. If $|A|=1$ then $|\b(A)|=d$. If $|A|>1$,
  let $w$ be a vertex in $A$ which is farthest from the root and let
  $A'=A\setminus\{w\}$. By the induction hypothesis and our choice of $w$,
  $|\b(A)|\ge |\b(A')| - 1 + (d-1) > (d-2)|A'|+d-2 = (d-2)|A|$.
\end{proof}
The above claim applies also to sets $A\in\A_v$ since subsets $A$ of
vertices of $\T_h^d$ which satisfy $(A\cup\b(A))\cap V_L=\emptyset$
have the same boundary $\b(A)$ in both $\T_h^d$ and
$\tilde{\T}_h^d$. We continue to define $\A_{v,\alpha}$ and
$b_\alpha$ exactly as in Section~\ref{sec:bound-omega} and note
that even though $v_0$ has very high degree in $\tilde{\T}_h^d$,
it is still true that $|\A_{v,\alpha}|\le d^{4\alpha}$ as in
\eqref{eq:Ava-count}.

Combining \eqref{eq:tilde_omega_+_bound}, the bound on
$|\A_{v,\alpha}|$ and the above claim, and using a similar
calculation to \eqref{eq:P-Ava}, we have
\begin{align}
  \label{eq:grounded_P-Ava}
  \sum_{A \in \A_{v,a}} \P(\tilde{\Omega}_A^+) &
  \leq d^{4\alpha} \cdot M (2\alpha+1)(2M+1)^{\alpha}\left(\frac{M}{M+1}\right)^{b_\alpha} \\
  &\le \exp\left(-\frac{b_\alpha}{M+1}+2\log(9Md^2)\alpha\right) \\
  &\le \exp \left( - b_\alpha \left(  \frac{1}{M+1} -  \frac{2 \log (9Md^2)}{d-2} \right) \right) \le  \exp \left( - \frac{b_\alpha}{2(M+1)} \right),\\
  \nonumber
\end{align}
where the last inequality follows by noting that our assumption that
$d\ge 40(M+1)\log(M+1)$ implies
\begin{align*}
  4(M+1)\log(9Md^2)+2 &\le
  5(M+1)\log(9Md^2)=5(M+1)\log(9M)+10(M+1)\log d\\
  &\le 5(M+1)\log(9M) + \frac{d}{2}\le 20(M+1)\log(M+1)+\frac{d}{2}\le d.
\end{align*}
We continue exactly as in Section~\ref{sec:bound-omega}, using that
$b_\alpha>(d-2)\alpha$ and our assumption that $d\ge
40(M+1)\log(M+1)$, and obtain
\begin{equation*}
  \P(\Omega)\, \le\,
  \exp\left(-\frac{|B_{\tilde{\T}_h^d}(v,t)|}{5(M+1)}\right)
\end{equation*}
where $B_{\tilde{T}_h^d}(v,t)$ denotes the graph ball of radius $t$
around $v$ in $\tilde{T}_h^d$. It remains to note that if
$\dist_{\T_h^d}(v,V_L)\le t$ then $\P(\Omega)$ is trivially zero
since $f$ is $M$-Lipschitz, so that we can replace
$B_{\tilde{T}_h^d}(v,t)$ by $B_{\T_h^d}(v,t)$. Finally we note that
when $\dist_{\T_h^d}(v,V_L)\ge t$ we have $|B_{{\T}_h^d}(v,t)|\ge
d(d-1)^{t-1}$.

\section{Homomorphisms}

\label{sec:hom}

Assume that $G$ is a $d$-regular $2n$-vertex $\lambda$-bi-expander
that is good ($\lambda \ll d/\log d$). Let $V_0$ and $V_1$ be the
two color classes of $G$ and let $v$ and $v_0$ be two (not
necessarily distinct) vertices of $G$. Without loss of generality,
we assume that $v_0 \in V_0$. Recall the definition of $\phase$ from
Lemma~\ref{lemma:phase_hom}. We estimate the probability that, given
an integer $t \geq 2$, a uniformly chosen random function $f \in_R
\Hom_{v_0}(G)$ is in the event
\[
\Omega = \{f \in \Hom_{v_0}(G) \colon |f(v) - \phase(f)| > t\}.
\]
Our proof of Theorem~\ref{thm:phase_dev_hom} closely follows the proof of Theorem~\ref{thm:phase_dev} given in Section~\ref{sec:Lipschitz}. We construct a map $T \colon \Omega \to \cP(\Hom_{v_0}(G))$ such that the set $T(f)$ is large for every $f \in \Omega$ and the set $\{f \in \Omega \colon g \in T(f) \}$ is small for every $g$.
We then derive a bound on $\P(\Omega)$ using Lemma~\ref{lemma:double-counting} and some easy counting. We start by splitting $\Omega$ into two parts:
\[
\Omega^+ = \{f \in \Omega \colon f(v) > \phase(f) + t\}
\quad \text{and} \quad \Omega^- = \{f \in \Omega \colon f(v) < \phase(f) - t\} .
\]
By definition of phase,
since $G$ is a good expander,
for every $f \in \Omega$
there is exactly one $k$ satisfying \eqref{eq:phase_def_hom}
with $i^*$.
This implies that the map $f \mapsto -f$ is one-to-one from
$\Omega^+$ to $\Omega^-$,
and vice versa.
We can, therefore, focus on the event $\Omega^+$.

\subsection{Defining the transformation $T$}

\label{sec:map-t-hom}

For $f\in\Omega^+$, write $k=k(f)=\phase(f)$ and let $A(f)$ be the
connected component of the vertex $v$ in the subgraph of $G^{\leq
2}$ induced by the set of vertices $\{w \colon f(w) > k+1\}$.
We further partition the event $\Omega^+$. Let $\C_v$ be the family of all $A \subseteq V(G)$ such that $v \in A$ and $A$ is connected in $G^{\leq 2}$. For every $A \in \C_v$, let
\[
\Omega_A^+ = \{f \in \Omega^+ \colon A(f) = A\}.
\]
\begin{claim}
  \label{cl:f_on_A-hom}
  Let $A \in \C_v$ and $f \in \Omega_A^+$. Set $X=\b(A)$. Then the following properties hold:
  \begin{enumerate}
  \item \label{item:fA-hom}
$\min f(A) > k+1$,
  \item \label{item:fX-hom}
   $f(X) = \{k+1\}$, and
  \item \label{item:fd2A-hom}
   $f (\b^2(A)) = \{k\}$.
  \end{enumerate}
\end{claim}
The proof is immediate from the definition of $A(f)$ and the
definition of homomorphism height functions. Here is a reasonable
guess on how to define $T(f)$ for $f\in\Omega^+$. Write $A=A(f)$ and
$X=\b(A)$, and for any $s \in \{-1,1\}^X$, define $h_s \colon V(G)
\to \Z$ by
\[
h_s(w) =
\begin{cases}
  f(w) - 2 & w \in A, \\
  k + s_w & w \in X, \\
  f(w) & w \not\in A \cup X,
\end{cases}
\]
The following statement is a direct consequence of
Claim~\ref{cl:f_on_A-hom}. We omit its proof, which is a simple case
analysis similar to the proof of Claim~\ref{claim:S}.

\begin{claim}
  \label{claim:varepsilon}
  For every $f\in\Omega^+$ and $s \in \{-1,1\}^X$, the function $h_s$ is a homomorphism function (i.e., $|h_s(v)-h_s(w)|=1$ whenever $v$ and $w$ are adjacent in $G$, but there is no guarantee that $h_s(v_0)=0$).
\end{claim}

Again, we use the ``shift'' operator $P_{v_0}$
defined by $P_{v_0}(h) = h - h(v_0)$.
Recall that $h \in \Hom(G)$ if and only if $P_{v_0}(h) \in \Hom_{v_0}(G)$.

\begin{claim}
  \label{claim:one-to-one-hom}
  For every $f\in\Omega^+$, $P_{v_0}$ is one-to-one on $\{h_s \colon s \in \{-1,1\}^X\}$.
\end{claim}
\begin{proof}
  Let $s \in \{-1,1\}^X$ and let $w \in A \cap N(X)$.
  Claim~\ref{cl:f_on_A-hom}, since $f$ is a homomorphism, tells us that
  $f(w) = k+2$.
  So, $h_s(w) = f(w) - 2 = k$.
  Since $h_s - P_{v_0}(h_s)$ is a constant function,
  \[
  h_s = P_{v_0}(h_s) + (k - P_{v_0}(h_s)(w)).
  \qedhere
  \]
\end{proof}

Finally, define $T \colon \Omega^+ \to \cP(\Hom_{v_0}(G))$ by
\begin{equation}
  \label{eq:Tf-def-hom}
  T(f) =  \left\{P_{v_0}(h_s) \colon  s \in \{-1,1\}^X \right\}.
\end{equation}

\subsection{Properties of $T$}

\begin{lemma}
  \label{lemma:hom-main}
  For every $A \in \C_v$, the following holds:
  \begin{enumerate}
  \item \label{item:hom-main-1}
    If $f \in \Omega_A^+$, then $|T(f)| = 2^{|\b(A)|}$.
  \item \label{item:hom-main-2}
    For every $h \in \Hom_{v_0}(G)$, we have
    \[
    |\{f \in \Omega_A^+ \colon h \in T(f)\}| \leq 2.
    \]
  \end{enumerate}
\end{lemma}
\begin{proof}
  {\it \ref{item:hom-main-1}.}
  Follows directly from~\eqref{eq:Tf-def-hom} and Claim~\ref{claim:one-to-one-hom}.

{\it \ref{item:hom-main-2}.}
  Let $X = \b(A)$ and fix some $h \in \Hom_{v_0}(G)$ and an integer $k$. We claim that there is at most one $f \in \Omega_A^+$ such that $h \in T(f)$ and $\phase(f) = k$. First, since $f \in \Omega_A^+$, we have $f(w) = k+1$ for all $w \in X$. Next, recall that $h = P_{v_0}(h_s)$ for some $s \in \{-1,1\}^X$, that $f(w) = h_s(w)$ for all $w \not\in A \cup X$,
  and that $f(w) = h_s(w) + 2$ for all $w \in A$. Hence, it suffices to reconstruct $h_s$ which amounts to reconstructing $h_s(v_0)$. For this, fix an arbitrary $w \in A \cap \n(X)$ and observe, as in the proof of Claim~\ref{claim:one-to-one-hom}, that $h_s(v_0) = h_s(w) - h(w) = k - h(w)$.

 Therefore, in order to bound the number of possibilities for $f$, it suffices to bound the number of integers $k$ for which there exists an $f$ as above.

  Recall that $h_s(w) = k = h_s(v_0) + h(w)$ for any $w \in A \cap \n(X)$. There are two cases.
  (i) If $v_0 \notin A \cup X$, then $h_s(v_0) = f(v_0) = 0$ and so $k$ is determined by $h$.
  (ii) If $v_0 \in A \cup X$, then $h_s(v_0)$ can only take values $0$ and $-2$, and hence $k$ can take at most $2$ values.
\end{proof}

Concluding this part of the discussion,
Lemmas~\ref{lemma:double-counting} and~\ref{lemma:hom-main}
imply the following bound on the probability of $\Omega_A^+$.

\begin{cor}
  \label{cor:POmegaA+-hom}
  For every $A \in \C_v$,
  \[
  \P(\Omega_A^+) \leq 2^{1-|\b(A)|}.
  \]
\end{cor}

\subsection{Deriving the bound on $\P(\Omega)$}

\label{sec:bound-omega-hom}

\begin{proof}[Proof of Theorem \ref{thm:phase_dev_hom}]
Fix some $f \in \Omega^+$ and let $k = \phase(f)$.
Recall the definition of $A(f)$ from Section~\ref{sec:map-t-hom} and observe that by~\eqref{eq:phase_hom-cor} in Lemma~\ref{lemma:phase_hom},
\[
|A(f)| \leq |\{w \colon f(w) > k + 1\}| \leq |\{w \colon |f(w) - k| \geq 2 \}| \leq \frac{3\lambda n}{d}.
\]
Since $f(v) > k + t$ and $f$ is a homomorphism,
$B(v,t-1) \subseteq A(f)$.
Let
\[
\A_v = \left\{A \in \C_v \colon \text{$B(v,t-1) \subseteq A$ and $|A| \leq \frac{3\lambda n}{d}$}\right\}.
\]
Clearly, the sets $\Omega_A^+$ defined in Section~\ref{sec:map-t-hom}, with $A \in \A_v$, form a partition of $\Omega^+$.

We are now ready to derive our bound on $\P(\Omega)$. For an integer $\alpha$, let
\[
\A_{v,\alpha} = \{A \in \A_v \colon |A| = \alpha \}.
\]
Since every $A \in \A_v$ is connected in $G^{\leq 2}$ and the maximum degree of $G^{\leq 2}$ is at most $d^2$,
by Lemma~\ref{lemma:connected-sets},
\begin{equation}
  \label{eq:Ava-count-hom}
  |\A_{v,\alpha}| \leq d^{4\alpha}.
\end{equation}
Let $\alpha$ be such that $\A_{v,\alpha}$ is
non-empty. In particular, $|B(v,t-1)| \leq \alpha \leq \frac{3\lambda n}{d}$.
Let
\[
b_\alpha = \min\left\{|\b(A)| \colon A \in \A_{v,\alpha} \right\} .
\]
By Corollary~\ref{cor:bi-exp-bdry-growth} and since $\alpha\le
\frac{3\lambda n}{d}$ and $\lambda \leq \frac{d}{12}$,
\begin{equation}
  \label{eq:partial-a-hom}
  b_\alpha \geq \min\left\{\frac{n}{4}, \frac{d^2}{5\lambda^2} \alpha \right\} \geq \frac{d}{12 \lambda} \alpha .
\end{equation}
Hence, by Corollary~\ref{cor:POmegaA+-hom},
\eqref{eq:Ava-count-hom},\eqref{eq:partial-a-hom} and since $G$ is a
good bi-expander, for every $\alpha$,
\begin{equation}
  \label{eq:P-Ava-hom}
  \sum_{A \in \A_{v,\alpha}} \P(\Omega_A^+) \leq d^{4\alpha} \cdot
  2^{1-b_\alpha} \leq
  \exp\left(4\alpha\log d - (b_\alpha-1)\log 2\right) \leq 2\exp(-b_\alpha/2).
\end{equation}
Finally, a computation analogous to the one given in the end of Section~\ref{sec:bound-omega} shows that~\eqref{eq:P-Ava-hom} implies
\[
\P(\Omega) \leq 2  \P(\Omega^+) \leq  \exp\left(-\frac{|B(v,t)|}{3}\right).
\qedhere
\]
\end{proof}

\section{Summary and future work}
\label{sec: sum}

This work investigates the typical behavior of Lipschitz functions
on expander graphs. The general understanding obtained is that if
the graph is a ``good'' expander (in an exact numerical sense) then
a random Lipschitz function on it is unlikely to fluctuate much at
any given vertex and is unlikely to have its maximum value larger
than $\log(\log n)$, where $n$ is the number of vertices in~$G$.
Similar results are obtained also for homomorphism height functions
on bi-expanders and for grounded Lipschitz functions on regular
trees.

There are 3 parameters influencing our results: the regularity $d$
of the graph, the expansion parameter $\lambda$ of the graph, and
the maximal slope $M$ of our functions. Our results for Lipschitz
functions hold under the assumption that
\begin{equation*}
  \lambda \leq \frac{d}{32(M+1)\log (9Md^2)}.
\end{equation*}
It is natural to ask how sharp is this condition. For example, do
similar results continue to hold when $\lambda$ is only slightly
less than $d$ (e.g., $\lambda=(1-\epsilon)d$)? or when $d$ is very
small (e.g., 3-regular expanders)? In another direction, recalling
that $\lambda$ cannot be significantly smaller than $\sqrt{d}$, we
see that even on the best of expanders, our results are limited to
$M<O(\sqrt{d}/\log d)$. Do similar results continue to hold for
larger $M$? Here, one may ask the same question for the limiting
continuous model (as $M\to\infty$) in which one samples a random
continuous-Lipschitz function $f:V(G)\to\R$, i.e., a uniform
function (in the sense of Lebesgue measure) from the set of
\emph{real-valued} functions satisfying $|f(v)-f(w)|\le 1$ for
adjacent $v$ and $w$, and $f(v_0)=0$ for some fixed vertex $v_0$.

We remark that our results for grounded Lipschitz functions on trees
apply under the weaker assumption that $M<O(d/\log d)$, see Theorem~\ref{thm:grounded}.
In a subsequent work~\cite{trees}, using a different approach
than the one used here,
we will show that variants of our theorem for
$d$-regular trees continue to hold for arbitrary $M$.

Investigating random Lipschitz functions is of interest on many
graphs. Can one obtain general necessary and sufficient conditions
on the graph for typical Lipschitz functions to be flat? As in the
theory of random surfaces, can the behavior of random Lipschitz
functions be related to the behavior of the Gaussian free field (see
also \cite{BenSch})? As mentioned already in the introduction, the
behavior of random homomorphism functions is now reasonably
understood on the hypercube~\cite{Gal, Kah2} and on (finite
boxes of) the lattice $\Z^d$ with large $d$ \cite{Pel}. In addition,
Engbers and Galvin \cite{EngGal2, EngGal1} use entropy methods (first
introduced by Kahn \cite{Kah1}) to obtain very general results on
graph homomorphisms on hypercube and certain bipartite
graphs. In a subsequent work~\cite{Hhom}, we will show how the methods
developed in this paper can be adapted to yield similar results about general graph
homomorphisms on expanders. It seems natural to ask what
the entropy methods used by Engbers, Galvin, and Kahn would yield
for the fluctuations of $M$-Lipschitz functions.

Lastly, we mention the tantalizing question of understanding random
homomorphism and Lipschitz functions on the two-dimensional lattice
$\Z^2$. It is conjectured that flatness no longer holds for this
graph but very little is known. Investigating this model is related
to well-known models of statistical physics such as square-ice, the
6-vertex model and 3-colorings (the antiferromagnetic 3-state potts
model). See \cite{Pel} for a discussion and simulation results.

\paragraph*{Acknowledgement}

We thank Itai Benjamini for suggesting 
the problem of understanding the typical behavior of
Lipschitz functions on expander graphs.

\bibliographystyle{amsplain}


\begin{thebibliography}{10}

\bibitem{BenHagMos}
I.~Benjamini, O.~H{\"a}ggstr{\"o}m, and E.~Mossel, \emph{On random graph
  homomorphisms into {$\mathbb{Z}$}}, J. Combin. Theory Ser. B \textbf{78}
  (2000), 86--114.

\bibitem{BenSch}
I.~Benjamini and G.~Schechtman, \emph{Upper bounds on the height difference of
  the {G}aussian random field and the range of random graph homomorphisms into
  {$\mathbb{Z}$}}, Random Structures Algorithms \textbf{17} (2000), no.~1,
  20--25.

\bibitem{BenYadYeh}
I.~Benjamini, A.~Yadin, and A.~Yehudayoff, \emph{Random graph-homomorphisms and
  logarithmic degree}, Electron. J. Probab. \textbf{12} (2007), 926--950.

\bibitem{EngGal2}
J.~Engbers and D.~Galvin, \emph{H-colouring bipartite graphs}, Arxiv:1101.0839,
  2011.

\bibitem{EngGal1}
\bysame, \emph{H-colouring tori}, Arxiv:1101.0840, 2011.

\bibitem{Fr}
J.~Friedman, \emph{A proof of {A}lon's second eigenvalue conjecture and related
  problems}, Memoirs of the American Mathematical Society \textbf{195} (2008),
  viii+100.

\bibitem{Gal}
D.~Galvin, \emph{On homomorphisms from the {H}amming cube to {$\mathbb{Z}$}},
  Israel J. Math. \textbf{138} (2003), 189--213.

\bibitem{HoLiWi}
S.~Hoory, N.~Linial, and A.~Wigderson, \emph{Expander graphs and their
  applications}, Bull. Amer. Math. Soc. (N.S.) \textbf{43} (2006), 439--561
  (electronic).

\bibitem{Kah1}
J.~Kahn, \emph{An entropy approach to the hard-core model on bipartite graphs},
  Combin. Probab. Comput. \textbf{10} (2001), no.~3, 219--237.

\bibitem{Kah2}
\bysame, \emph{Range of cube-indexed random walk}, Israel J. Math. \textbf{124}
  (2001), 189--201.

\bibitem{Pel}
R.~Peled, \emph{High-dimensional {L}ipschitz functions are typically flat},
  Arxiv:1005.4636, 2010.

\bibitem{trees}
R.~Peled, W.~Samotij, and A.~Yehudayoff, \emph{Grounded {L}ipschitz functions
  on trees}, in preparation.

\bibitem{Hhom}
\bysame, \emph{{$H$}-coloring expander graphs}, in preparation.

\end{thebibliography}
\providecommand{\bysame}{\leavevmode\hbox to3em{\hrulefill}\thinspace}
\providecommand{\MR}{\relax\ifhmode\unskip\space\fi MR }
\providecommand{\MRhref}[2]{%
  \href{http://www.ams.org/mathscinet-getitem?mr=#1}{#2}
}
\providecommand{\href}[2]{#2}

\end{document}